\documentclass{imsart}

\usepackage{amsmath}
\usepackage{amsthm}
\usepackage{amssymb}

\usepackage{mhequ}
\usepackage{graphicx}
\usepackage{placeins}
\usepackage{xfrac}
\usepackage{times}
\usepackage{mathrsfs}
\usepackage[numbers]{natbib}
\usepackage[colorlinks]{hyperref} 

\startlocaldefs
\theoremstyle{plain}

\newtheorem{theorem}{Theorem}[section]

\newtheorem{assumption}[theorem]{Assumption}

\newtheorem{lemma}[theorem]{Lemma}
\newtheorem{definition}[theorem]{Definition}

\renewcommand{\P}{\mathcal P}
\newcommand{\X}{\mathbf X}

\newcommand{\be}{\begin{equs}}
\newcommand{\ee}{\end{equs}}
\newcommand{\vertiii}[1]{{\left\vert\kern-0.25ex\left\vert\kern-0.25ex\left\vert #1 
    \right\vert\kern-0.25ex\right\vert\kern-0.25ex\right\vert}}
\newcommand{\lvertiii}{{\vert\kern-0.25ex\vert\kern-0.25ex\vert}}    
\newcommand{\rvertiii}{{\vert\kern-0.25ex\vert\kern-0.25ex\vert}}
\newcommand{\1}{\mathbf 1}    
\newcommand{\R}{\mathbb R}
\newcommand{\bb}[1]{\mathbb #1}
\newcommand{\eps}{\epsilon}

\DeclareMathOperator{\esssup}{esssup}
\DeclareMathOperator{\TV}{TV}

\newcommand{\ci}{\perp \! \! \! \perp}

\renewcommand{\S}{\mathcal S}

\renewcommand{\Pr}{\mathbf P}

\newcommand{\trel}{\tau_{\mathrm{rel}}}
\newcommand{\thit}{\tau_{\mathrm{hit}}}

\newcommand{\E}{\mathbf E}

\endlocaldefs
    
\begin{document}
\begin{frontmatter}

\title{Fast Mixing of Metropolis-Hastings with Unimodal Targets}
\runtitle{Unimodal Targets}

\author{\fnms{James E.} \snm{Johndrow}\corref{} \ead[label=e1]{johndrow@stanford.edu}}
\address{390 Serra Mall \\ Stanford, CA, USA \\ \printead{e1}}
\affiliation{Stanford University}
\and
\author{\fnms{Aaron} \snm{Smith}\ead[label=e2]{asmi28@uottawa.ca}}
\address{585 King Edward Drive \\ Ottawa, ON, Canada \\ \printead{e2}}
\affiliation{University of Ottawa}

\runauthor{J. Johndrow and A. Smith}

\begin{abstract}
A well-known folklore result in the MCMC community is that the Metropolis-Hastings algorithm mixes quickly for any unimodal target, as long as the tails are not too heavy. Although we've heard this fact stated many times in conversation, we are not aware of any quantitative statement of this result in the literature, and we are not aware of any quick derivation from well-known results. The present paper patches this small gap in the literature, providing a generic bound based on the popular ``drift-and-minorization'' framework of \citet{rosenthal1995minorization}. Our main contribution is to study two sublevel sets of the Lyapunov function and use path arguments in order to obtain a sharper general bound than what can typically be obtained from multistep minorization arguments.
\end{abstract}

\begin{keyword}[class=MSC60J05]
\kwd{Metropolis-Hastings, Markov chain Monte Carlo, Spectral Gap}
\end{keyword}
\end{frontmatter}

\section{Introduction}
The Metropolis algorithm \citep{metropolis1953equation} and its generalization, the Metropolis-Hastings algorithm \citep{hastings1970monte}, have been exceptionally successful in the numerical approximation of analytically intractable integrals. Because these algorithms are both important and difficult to analyze, there is an enormous literature on the properties of Metropolis-Hastings chains in the statistics, computer science, mathematics and physics communities (see \textit{e.g.} the popular textbooks \cite{meyn1993markov,levin2009markov}). Despite the size of this literature, obtaining reasonable quantitative bounds on the convergence rates of specific Markov chains used in statistics can be quite difficult, even when there are good heuristic reasons that convergence should be quick \cite{jones2001honest,diaconis2008gibbs}. Recently, the authors needed to use an ``obvious" folklore result that does not seem to be in the literature: reasonable Metropolis-Hastings chains targetting unimodal distributions will mix quickly. The main purpose of the paper is to provide a general and quantitatively useful version of this folklore result (see Theorem \ref{ThmMainThm}). 

We were originally motivated by the need to prove a sharp lower bound on the spectral gap of a Metropolis-Hastings algorithm for logistic regression in a ``rare-success" asymptotic regime (see \citet{johndrow2016mcmc}). When the standard deviation of the proposal kernel was similar to the standard deviation of the target distribution, it was straightforward to obtain a quantitatively strong version of the ``minorization" condition required by the ``drift-and-minorization" approach of \citet{rosenthal1995minorization}. However, this argument becomes much more delicate when the proposal variance does not closely match the target variance. Although we were motivated by a specific problem, similar problems appear more generally when the proposal kernel of an MCMC algorithm is not perfectly tuned to the target. This sort of (initial) bad tuning can be difficult to avoid in contexts such as \citet{johndrow2016mcmc} where the posterior distribution is very far from Gaussian.

To address this technical problem, we combine pathwise arguments (as studied in \cite{yuen2000applications}) with coupling arguments to obtain reasonable estimates of mixing times inside of compact sublevel sets of the Lyapunov function. We then apply the ``drift and minorization" approach of \cite{rosenthal1995minorization} to obtain mixing bounds on the full state space. 
This argument is presented here for generic random-walk type Metropolis-Hastings, and thus should be broadly useful.

\subsection{Related Work}

Popular approaches for establishing bounds on convergence rates for Markov chains include the Lyapunov-small set techniques of \cite{khasminskii2011stochastic, meyn1993markov, rosenthal1995minorization}, and geometric inequalities such as Poincar\'{e}, Cheeger, and log-Sobolev inequalities \cite{diaconis1993comparison, diaconis1996nash, diaconis1996logarithmic, lawler1988bounds, sinclair1989approximate,yuen2000applications}.  Under suitable conditions on the tails of the target and the proposal kernel, drift and minorization arguments show that the Metropolis-Hastings algorithm will converge to the target at an exponential rate \cite{mengersen1996rates, jarner2000geometric, jarner2003necessary}.  The paper \cite{jarner2004conductance} studies essentially the same question addressed in the present paper using Cheeger inequalities, but restricts their attention only to log-concave target distributions.

\section{Notation and Standing Assumptions}

Consider a Markov kernel $\P$ with a unique invariant measure $\mu : \mu \P = \mu$. The spectrum of $\P$ is the set $S$
\be
S(\P) = \{ \lambda \in \bb C \setminus \{0\} : (\lambda I - \P)^{-1} \text{ is not a bounded linear operator on } L^2(\mu) \}
\ee
and the spectral gap 
\be
\alpha = 1 - \sup \{ |\lambda| \, : \, \lambda \in S, \, \lambda \ne 1 \}
\ee
when the eigenvalue $1$ has multiplicity 1, and $\lambda^*(\P) = 0$ else. Define the relaxation time $\trel \equiv \alpha^{-1}$, and the mixing time $\tau$ of $\P$ on a set $\Theta$,
\be
\tau \equiv \min \{t \, : \, \sup_{x \in \Theta} \, \|\delta_x \P^t - \mu \|_{\TV} < 1/4 \}, 
\ee
which need not be finite. 

The following is a strong notion of unimodality on a set:

\begin{definition}
Fix an interval $[a,b] \subset \mathbb{R}$. Call a function $f \, : \, [a,b] \rightarrow \mathbb{R}$ \textit{unimodal} with mode $m \in I$ if $f$ is monotonely increasing on $[a,m]$ and monotonely decreasing on $[m,b]$.
\end{definition}

\begin{definition} \label{def:Unimodal}
Fix a convex subset $\Theta \subset \mathbb{R}^{d}$. Call a function $f \, : \, \Theta \rightarrow \mathbb{R}$ \textit{multivariate unimodal} if, for all $x \in \Theta$ and $v \in \mathbb{R}^{d}$, the function $f_{x,v}(s) \equiv f(x + sv)$ is unimodal. 

Note that a multivariate unimodal function $f$ will have a (possibly non-unique) maximum $M$; we call any point $m$ satisfying $f(m) = M$ a \textit{mode} of $f$.
\end{definition}

Throughout the remainder of the paper, we fix scale $\epsilon > 0$ for typical step sizes of $\P$, in a way that is made concrete in the context of the following assumptions on $\P$:

\begin{assumption} \label{ass:MHKernel}
 The Markov kernel of interest $\P$ is a Metropolis-Hastings kernel with target distribution $\mu$ and proposal kernel $Q$ that satisfies
 \begin{enumerate}
  \item $Q(x,\cdot)$ has density $q(x,\cdot)$ and $\mu$ has density $p(\cdot)$ with respect to Lebesgue measure.
  \item $q$ is isotropic, i.e. it is of the form $q(x,y) = q(\|x-y\|)$ for some density $q$ that is unimodal.
  \item For some $\delta_{1}, c_{1}, c_{2} > 0$, $q$ satisfies 
\be \label{IneqFunnyQBounds}
q(x) & \geq \delta_{1} \1_{\|x \| \leq \epsilon} \\
q(x) & \leq c_{1} e^{-\frac{c_{2}\|x\|}{\epsilon}}.
\ee 

  \item There exist constants $\gamma \in (0,1)$ and $0 \le K < \infty$ and a Lyapunov function $V : \R^d \to [0,\infty)$ satisfying
  \be \label{IneqGenLyap}
  (\P V)(x) \le \gamma V(x) + K.
  \ee
 \end{enumerate}
\end{assumption}

The first three assumptions hold for most Metropolis-Hastings proposal kernels used in practice, and we expect them to be easy to verify. The last condition is stronger, and it can be difficult to verify that the condition holds with reasonably small constants $\gamma, K$.  However, Lyapunov functions do exist under fairly mild conditions that have been well-studied (see \textit{e.g.} \citep{mengersen1996rates}, \citep{jarner2000geometric}).

For chains of this form, define the Metropolis-Hastings acceptance probability by
\be
\alpha(x,y) \equiv 1 \wedge \frac{p(y) q(y,x)}{p(x) q(x,y)}.
\ee

\section{Main Result}
For any $\Theta \subset \R^d$ with $\mu(\Theta) > 0$, denote by $\mu_{\Theta}$ and $p_{\Theta}$ the usual restrictions of $\mu, p$ to $\Theta$:
\be 
\mu_{\Theta}(A) = \frac{\mu(\Theta \cap A)}{\mu(\Theta)}, \quad p_{\Theta}(x) = \frac{1}{\mu(\Theta)} p(x) \textbf{1}_{x \in \Theta}.
\ee 
Similarly, denote by $\P_{\Theta}$ the usual restriction of $\P$ to $\Theta$. That is, $\P_\Theta$ is a Metropolis-Hastings chain with proposal kernel $Q$ and target distribution $\mu_{\Theta}$. Denote by $B_r(x)$ a Euclidean ball of radius $r$ centered at $x \in \R^d$. Our main result is the following

\begin{theorem} \label{ThmMainThm}
Let $\P$ be a Metropolis-Hastings transition kernel on $\mathbb{R}$ satisfying the conditions of Assumption \ref{ass:MHKernel}. 
Let $\Theta \subset \R$ be a set satisfying
\begin{enumerate}
 \item $0<p<\infty$ is unimodal on $\Theta$ with mode $m \in \Theta$, and $\Theta \subseteq B_L(m)$ for some constant $L > 0$.
 \item $p$ satisfies 
 \be \label{EqNearlyUniform}
 \inf_{x \in B_{2 \epsilon}(m)} p(x) > \frac{15}{16} p(m),
 \ee 
i.e. it is ``almost constant'' on a ball of radius $2 \epsilon$ around the mode.
\end{enumerate}
Then there exists a constant $C = C(c_{1},c_{2}, \delta_{1}) < \infty$ such that the mixing time $\tau$ of $\P_\Theta$ satisfies 
\be \label{IneqMainThmConc1}
\tau < C \epsilon^{-3} \delta_1^{-1} L^4 p_\Theta(m).
\ee
If the set $\Theta$ is ``small but not too small'', i.e.
 \be \label{eq:ThetaMinSize}
 \Theta \supset \left\{ x \in \mathbb{R}^d \, : \, V(x) \leq \frac{8}{1-\gamma} \left( \frac{4K}{1-\gamma} + K C \epsilon^{-2} \delta_1^{-1} L^3 p_\Theta(m) \right) \right\},
 \ee
we also have
\be \label{IneqMainThmConc2}
\trel(\P) &\le C \epsilon^{-3} \delta_1^{-1} L^4 p_\Theta(m). 
\ee
\end{theorem}
Although our final result is restricted to $\R$, several of the Lemmas used in proving the result hold with almost no changes in $\R^d$ and could be used to prove similar results for higher-dimensional target distributions. Thus, we prove most of the results in $\R^d$ and specialize to the case of $\R$ for the final Lemma.

\section{Proofs}

We break the proof up into three lemmas, each of which might be individually useful for proving similar results. The first two lemmas are proved on $\R^d$; the final lemma is proved only on $\R$.

\subsection{Mixing: From Very Small Sets to Small Sets}
The first Lemma shows that a Lyapunov condition combined with a bound on the mixing time $\tau$ of $\P_\Theta$ for a ``small'' sublevel set $\Theta$ of $V$ allows us to bound the spectral gap by the inverse of the mixing time. The key idea is that a Markov chain started from a point $x$ inside of a ``very small'' sublevel set is unlikely to escape from the slightly larger ``small" set within its first $\tau$ steps. This allows us to minorize $\P^\tau$ by $\mu_\Theta$ inside of the ``very small'' set $\{ x : V(x) < 4K (1-\gamma)^{-1}\}$ and then apply the usual Harris theorem to obtain a bound on the spectral gap. In essence we use the Lyapunov condition and the mixing time bound to obtain a minorization condition on the time scale of the mixing time -- that is, for $\P^\tau$ -- rather than directly showing a multistep minorization condition.

\begin{lemma} \label{Lemma1Stat}
Suppose that $V$ is a Lyapunov function of $\P$ satisfying \eqref{IneqGenLyap}, $\Theta \subset \R^d$ satisfies \eqref{eq:ThetaMinSize}, and the mixing time of $\P_{\Theta}$ is $\tau < \infty$. Then there exists $C = C(\gamma,K)$ independent of $\tau$ and $d$ so that the relaxation time of $\P$ is at most 
\be 
\trel(\P) \leq C \tau.
\ee 
\end{lemma}

\begin{proof}
Let
\be
\S(R) &= \{ x : V(x) \le R \}, \\
R_1 &= \frac{4K}{1-\gamma}, \quad R_2 = 8 \left( \frac{4K}{(1-\gamma)^2} + \frac{K \tau}{1-\gamma} \right)
\ee
and fix $x \in \S(R_1)$. Let $\{X_{t}\}_{t \geq 0}$ be a Markov chain with transition kernel $\P$ and initial state $X_0 = x$. Denote by $\kappa = \inf \{t : X_t \in \Theta^c\}$ the first hitting time of $\Theta^c = \{ x \in \R^d : x \notin \Theta\}$. By Inequality \eqref{IneqGenLyap} and Markov's inequality,
\be 
\Pr[ \kappa \le \tau ] &\leq \Pr[\max_{0 \leq k \leq \tau} V(X_k) \geq R_2] \le \sum_{k=0}^\tau \Pr\left[ V(X_t) > R_2 \right] \\
&\leq R_2^{-1} \sum_{k=0}^\tau \E[V(X_t)] \\
&\leq R_2^{-1} \sum_{k=0}^\tau \left( \gamma^k V(x) + K \sum_{j=0}^{k-1} \gamma^j \right) \\
&\leq R_2^{-1} \sum_{k=0}^\tau  \left( \gamma^k V(x) + K \frac{1-\gamma^{k+1}}{1-\gamma} \right) \\
&\leq R_2^{-1} \left(\frac{R_1}{1-\gamma} + K \frac{(1-\gamma) \tau +\gamma^{\tau+1}-\gamma}{(1-\gamma)^2} \right) \\
&\leq R_2^{-1} \left(\frac{R_1}{1-\gamma} + K \frac{\tau}{1-\gamma} \right) \leq \frac{1}{8}.
\ee 
 
Applying this bound, the maximal coupling inequality, and the triangle inequality, we obtain the minorization bound:
\be 
\sup_{x \in R_1} \| \delta_x \P^\tau - \mu_\Theta \|_{\TV} &\leq \sup_{x \in R_1} \| \delta_x \P_\Theta^\tau - \mu_{\Theta} \|_{\TV} + \sup_{x \in R_1} \|\delta_x \P^\tau - \delta_x \P^\tau_\Theta \|_{\TV} \\
&\leq \frac14 + \frac18 = \frac38.
\ee 

So then $\P^\tau$ satisfies
\be
\inf_{x \in \S(R_1)} \delta_x \P^\tau(\cdot) &\ge \frac38 \mu_\Theta(\cdot) \\
(\P^\tau V)(x) &\le \gamma^\tau V(x) + \frac{K(1-\gamma^\tau)}{1-\gamma}.
\ee

Applying this minorization bound and the Lyapunov bound \eqref{IneqGenLyap}
along with Theorem 1 of \citet{hairer2011yet} to the Markov operator $\P^\tau$ implies that we can bound the geometric convergence rate of convergence $\bar \alpha$  of $\P^\tau$ via
\be 
\bar \alpha &= \inf_{\alpha_0 \in (0,5/8)} \left(1-\frac58 + \alpha_0 \right) \vee \frac{2 + \frac{4K}{1-\gamma} \frac{\alpha_0(1-\gamma)}{K (1-\gamma^\tau)} \left( \gamma^\tau + \frac{2 K (1-\gamma^\tau) (1-\gamma)}{(1-\gamma) 4K } \right)}{2+ \frac{4 K}{1-\gamma} \frac{\alpha_0(1-\gamma)}{K (1-\gamma^\tau)} } \\
&= \inf_{\alpha_0 \in (0,5/8)} \left(1-\frac58 + \alpha_0\right) \vee \frac{2 + \frac{4 \alpha_0}{(1-\gamma^\tau)} \frac{ \gamma^\tau + 1}{2} }{2+ \frac{4 \alpha_0}{(1-\gamma^\tau)} } \\
&\le \inf_{\alpha_0 \in (0,5/8)} \left(1-\frac58 + \alpha_0\right) \vee \frac{2 + \frac{4 \alpha_0}{(1-\gamma)} \frac{ \gamma + 1}{2} }{2+ \frac{4 \alpha_0}{(1-\gamma)} } < 1, \label{IneqConfusingAlphaBound}
\ee
where the last line follows because the second term in the maximum is decreasing in $\tau$. 
This implies that the geometric convergence rate of $\P$ is at most $\bar \alpha^{1/\tau} < 1$, and so the $L_2(\mu)$ spectral gap of $\P$ is at least
\be
(1-\bar \alpha^{1/\tau})
\ee
by an application of Theorem 2 of \citet{roberts1997geometric}. Inspection of inequality \eqref{IneqConfusingAlphaBound} completes the proof.
\end{proof}

\subsection{Mixing for Unimodal Distributions on Compact Sets}
We now show the first of two Lemmas necessary to prove the mixing time bound inside of $\Theta$. This lemma shows that when started from an initial condition very close to the mode, $\P_\Theta$ will mix rapidly. Our approach is to compare $\P_\Theta$ to a chain with transition kernel $\tilde {\P}_\Theta(x,\cdot) = \mu_{\Theta}(\cdot)$ for all $x \in \Theta$ - this chain simply takes iid samples from its stationary measure. 

We can write the transition densities of $\P$ and $\tilde{\P}$ as:
\be
p_\Theta(x,y) &= \delta_x(y) \int (1-\alpha_\Theta(x,y)) Q(x,dy) + \alpha_\Theta(x,y) q(x,y) \\
\tilde p_\Theta(x,y) &=  p_\Theta(y),
\ee
where
\be
\alpha_\Theta(x,y) = \frac{p_\Theta(y) q(y,x)}{p_\Theta(x) q(x,y)}.
\ee
For $A \subset \Theta$, define the $A$-restricted mixing time 
\be
\tau_A(\P_\Theta) =  \min \{ t : \sup_{x \in A} \|\delta_x \P^t_\Theta - \mu_\Theta \|_{\TV} \le 1/8 \}.
\ee

\begin{lemma} \label{Lemma2Stat}
Suppose that $p$ satisfies \eqref{EqNearlyUniform} and $Q$ satisfies \eqref{IneqFunnyQBounds}. 
Then
\be 
\tau_{B_\epsilon(m)}(\P_\Theta) \leq C \log(16) \epsilon^{-3} \delta_1^{-1} L^{d+3} p_\Theta(m) \frac{\pi^{d/2} 3^{d+2}}{\Gamma(d/2+1)},
\ee 
where $0 < C < \infty$ is a universal constant.
\end{lemma}
\begin{proof}

Our main tool is Theorem 3.2 of \cite{yuen2000applications}. 
We recall the definition of the ``linear'' set of paths $\Gamma$, consisting of steps of length $\epsilon$, given in Section 2 of \cite{yuen2000applications}. For fixed $x,y \in \Theta$, define the length $b_{x,y} = \lceil \frac{\|x-y\|}{\epsilon} \rceil$. When $\frac{\|x-y\|}{\epsilon}$ is \textit{not} an integer, set 
\be
\gamma_{x,y}^{(i)}  &= x + (i-1) \epsilon (y-x), \quad 0 \leq i < b_{x,y} \\
\gamma_{x,y}^{(b_{x,y})} &= y \\
\gamma_{x,y} &= (\gamma_{x,y}^{(0)}, \ldots,\gamma_{x,y}^{(b_{x,y})}).
\ee
Then $\Gamma = \{\gamma_{x,y} : (x,y) \in \Theta,\,  \alpha_\Theta(u,v) q(u,v) > 0\}$ is the collection of all such paths of finite length. We say that a pair $(u,v) \in \Theta \times \Theta$ is an $i$th edge of the path $\gamma_{x,y}$ iff $u = \gamma_{x,y}^{(i-1)}$ and $v = \gamma_{x,y}^{(i)}$. Let $E_i$ be the collection of the $i$th edges of all paths $\gamma \in \Gamma$, and put $E = \bigcup_{i \in \bb N} E_i$. As shown in Section 2 of \cite{yuen2000applications}, the set of paths $\Gamma$ satisfies the regularity conditions of Theorem 3.2 of that paper and, for any $(u,v) \in E$, the associated Jacobian satisfies $J_{x,y}(u,v) = b_{x,y}^{d}$ (see \citet[page 5]{yuen2000applications} for details). 

Define $\xi(u,v) = \alpha_\Theta(u,v) q(u,v) p_\Theta(u)$ and for any $\gamma_{x,y} \in \Gamma$ put $\|\gamma_{x,y}\|_0 = \sum_{(u,v) \in \gamma_{x,y}} \xi(u,v)^0 = b_{x,y}$. Notice we can also view the comparison kernel $\tilde \P_\Theta$ as a Metropolis-Hastings kernel with acceptance probability $\tilde \alpha(x,y) = 1$ and proposal $\tilde q(x,y) = p_\Theta(y)$. To use  Theorem 3.2 of \cite{yuen2000applications}, we must bound the geometric constant 
\be
A_0(\Gamma) &= \esssup_{(u,v) \in E} \left\{ \xi(u,v)^{-1} \sum_{\gamma_{x,y} \ni (u,v)} \|\gamma_{x,y}\|_0 \tilde p(x) \tilde \alpha(x,y) \tilde q(x,y) |J_{x,y}(u,v)| \right\} \\
&= \esssup_{(u,v) \in E} \left\{ \xi(u,v)^{-1} \sum_{\gamma_{x,y} \ni (u,v)} b_{x,y} p_\Theta(x) p_\Theta(y) b_{x,y}^d \right\}.
\ee
Bounding below by the uniform proposal on $B_\epsilon(x)$ we have using \eqref{IneqFunnyQBounds} and the volume of a unit ball in $\R^d$
\be
\xi(u,v) \ge \left( 1 \wedge \frac{p_\Theta(v)}{p_\Theta(u)} \right) \frac{\Gamma(d/2+1)}{\pi^{d/2}} \epsilon^{-d} \delta_1 p_\Theta(u) \ge (p_\Theta(u) \wedge p_\Theta(v)) \frac{\Gamma(d/2+1)}{\pi^{d/2}} \epsilon^{-d} \delta_1
\ee
which is everywhere positive for any $(u,v) \in \gamma_{x,y}$ by the definition of $\Gamma$. Define $b = \max_{x,y} b_{x,y} \leq 2 L \epsilon^{-1} + 1$, we have
\be
A_0(\Gamma) &= \esssup_{(u,v) \in E} \left\{ \xi(u,v)^{-1} \sum_{\gamma_{x,y} \ni (u,v)} b_{x,y} p_\Theta(x) p_\Theta(y) b_{x,y}^d \right\} \\
&\le  \frac{\pi^{d/2}}{\Gamma(d/2+1)} \delta_1^{-1} \epsilon^d (p_\Theta(u) \wedge p_\Theta(v))^{-1} b^{3+d} p_\Theta(x) p_\Theta(y) \\
&\le \frac{\pi^{d/2} \epsilon^d}{\Gamma(d/2+1)} \delta_1^{-1} b^{3+d} \frac{p_\Theta(x) \wedge p_\Theta(y)}{p_\Theta(u) \wedge p_\Theta(v)} p_\Theta(m) \label{eq:Path1}
\ee
where in the second line we used the fact that there are at most $\ell$ starting points for paths of length $\ell$ that contain the edge $(u,v)$ and $\sum_{j=1}^{\ell} j = (\ell^2+\ell)/2 \le \ell^2$.
Since $p$ is multivariate unimodal (see Definition \ref{def:Unimodal}), we note that
\be
\frac{p_\Theta(x) \wedge p_\Theta(y)}{p_\Theta(u) \wedge p_\Theta(v)} \le 1
\ee
for any points $u,v$ on a linear path from $x$ to $y$, and therefore
\be
A_0(\Gamma) &\le \frac{\pi^{d/2} \epsilon^d}{\Gamma(d/2+1)} \delta_1^{-1} b^{3+d} p_\Theta(m). 
\ee
Thus, combining with \eqref{eq:Path1}, we obtain with $b = 2 \epsilon^{-1} L + 1$ the maximum length of a path consisting of steps of size $\epsilon$ connecting two points inside a ball of radius $L$
\be
A_0(\Gamma) &\le \frac{\pi^{d/2}}{\Gamma(d/2+1)} 3^{2+d} \epsilon^{-3} L^{d+3} \delta_1^{-1} p_\Theta(m), 
\ee
where we used that $\epsilon \le L$ so that $2\epsilon^{-1} L + 1 < 3 \epsilon^{-1} L$.
It follows that the spectral gap 
$\alpha$ of $\P_{\Theta}$ satisfies 
\be \label{IneqGenUniSpec}
\alpha \geq 3^{-(d+2)} \epsilon^3 \delta_1 L^{-(d+3)} p_\Theta(m)^{-1} \frac{\Gamma(d/2+1)}{\pi^{d/2}}.
\ee

By Inequality \eqref{EqNearlyUniform}, observe that we can write
\[
\delta_x \P_\Theta = \frac{15}{16} \mu_\Theta |_{B_{2 \epsilon}(m)} + \frac{1}{16} r_x
\]
for some ``remainder'' measure $r_x$. Applying Proposition 1.1 of \cite{yuen2000applications} and the bound \eqref{IneqGenUniSpec} on the spectral gap of $\P_\Theta$, this implies there exists some absolute constant $0 < C < \infty$ such that
\be \label{IneqQuickMixMiddle}
\sup_{x \in B_\epsilon(m) } \| \delta_x \P^{t+1}_\Theta - \mu_{\Theta} \|_{\TV} &\leq \sup_{x \in B_{\epsilon}(m)} \| \mu_{\Theta}|_{B_\epsilon(x)} \P^t_\Theta - \mu_{\Theta} \|_{\TV} + \frac{1}{16} \\
&\leq \frac18
\ee 
for all 
\be
t > C \log(16) \epsilon^{-3} \delta_1^{-1} L^{d+3} p_\Theta(m) \frac{\pi^{d/2} 3^{d+2}}{\Gamma(d/2+1)}.
\ee
\end{proof}

We prove the last lemma:

\begin{lemma} \label{Lemma3Stat}
 Suppose $\P_\Theta$ is a Metropolis-Hastings kernel on $\Theta \subset \R$ satisfying Assumption \ref{ass:MHKernel}.  Then there exists some constant $C = C(\delta_{1}, c_{1}, c_{2}) < \infty$ such that
 \be
 \tau(\P_\Theta) \le \tau_{[m-\epsilon,m+\epsilon]}(\P_\Theta) + C \frac{L^{2}}{\epsilon^{2}}.
 \ee
\end{lemma}

\begin{proof}
Define the function $F$ by
\be 
F(x,\Delta,U) = x + \Delta \mathbf 1\{U < \alpha_\Theta(x,x+\Delta)\}.
\ee 
Note that if $\Delta \sim Q(x,\cdot)$ and $U \sim \mathrm{Uniform}(0,1)$ with $\Delta \ci U$, then
\be
F(x,\Delta,U) \sim \P_\Theta(x,\cdot),
\ee
so $F$ defines a \textit{forward mapping representation} of $\P_{\Theta}$. Next, let $\Delta_t \stackrel{iid}{\sim} Q(0,\cdot)$ and $U_t \stackrel{iid}{\sim} \mathrm{Uniform}(0,1)$. We fix $x \in (m+\epsilon,m+L] \cap \Theta$ and consider a Markov chain $(X_t,Y_t,Z_t)$ on $\X \times \X \times \X$, with initial state $(x,x,x)$ and dynamics defined jointly by
\be
X_{t+1} &= F(X_t,\Delta_t,U_t) \\
Y_{t+1} &= Y_t + \Delta_t \mathbf 1(Y_t + \Delta_t \in [-L,L]) \\
Z_{t+1} &= Z_t + \Delta_t.
\ee
We focus initially on the properties of $(X_t,Y_t)$. Define
\be 
\thit^x(x) &= \inf \{t \geq 0 \,: \, X_t \in [m-\epsilon,m+\epsilon] \} \\
\thit^y(x) &= \inf \{t \geq 0 \,: \, Y_t \in [m-L,m+\epsilon] \}. 
\ee 
Defining $\tau^* = \tau^x(x) \wedge \tau^y(x)$, it is clear that $m+\epsilon \le X_t \le Y_t$ for all $t < \tau^*$.
We also have
\be
\Pr[ X_{t+1} \in [m,m+\eps] \mid X_t > m+\eps, t < \tau^*] &\geq \Pr[ Y_{t+1} \in [m,m+\eps] \mid Y_t > m+\eps, t < \tau^*]
\ee
and
\be
\Pr[ Y_{t+1} \in [m,m+\eps] \mid Y_{t+1} \in [m-L,m+\eps], Y_t > m+\eps, t < \tau^*] &> \delta \equiv \delta(c_{1}, c_{2}) > 0,
\ee
where the second bound comes from Inequality \eqref{IneqFunnyQBounds}. Combining these two bounds, we have shown
\be
\Pr[\thit^x(x) \le t] &\ge \delta \Pr[\thit^y(x) \le t].
\ee

Now we just need a bound on $\thit^y(x)$, which we obtain by comparing $Y_t$ and $Z_t$. By the Berry-Esseen theorem and the sub-exponential tail bound in Inequality  \eqref{IneqFunnyQBounds}, there exists a $C_1 = C_{1}(c_{1},c_{2},\delta_{1}) < \infty$ such that, for all  $t > C_1 \frac{L^2}{\epsilon^2}$,
\be
\Pr[Z_t < m] > \frac{1}{8}.
\ee
Using again the fact that the tails of $Q$ are sub-exponential (in the sense of Inequality \eqref{IneqFunnyQBounds}), along with the fact that $Y_t \le Z_t$ for all $t < \min \{s : Z_s < -L\}$, there exists $C_2 = C_{2}(c_{1},c_{2},\delta_{1}) > 0$ such that
\be
\Pr[\min_{0 \le s \le t} Y_s \le m+\epsilon] \ge C_2 \Pr[Z_t < m] \ge \frac{C_2}{8} \equiv \eta.
\ee
This implies
\be 
\Pr\left[ \thit^y(x) > C_3 \frac{L^2}{\eps^2} \right] &\le 1-\eta \\
\Pr\left[ \thit^x(x) \le C_3 \frac{L^2}{\eps^2} \right] &\ge \delta \eta;  \\
\ee 
for some constant $C_{3} > 0 $. Using the strong Markov property, we conclude that for all $T \in \mathbb{N}$
\be \label{EqHittingIneq}
\Pr\left[ \thit^x(x) > C_3 T \frac{L^2}{\epsilon^2} \right] \leq (1-\delta \eta)^{T}.
\ee 
We proved this inequality for $x \in [m+\epsilon,L] \cap \Theta$. By the symmetry of the situation, it is clear that this also holds for $x \in [-L, m-\epsilon] \cap \Theta$, and of course it trivially holds for $x \in [m-\epsilon,m+\epsilon] \cap \Theta$. Thus, Inequality \eqref{EqHittingIneq} holds for all $x \in \Theta$.

Combining Inequality \eqref{IneqQuickMixMiddle} and  Inequality \eqref{EqHittingIneq} (with the choice $T = (-\log (8))/\log(1-\delta \eta)$), and setting $t =  \tau_{[m-\epsilon,m+\epsilon]}(\P_{\Theta}) + C_{3} T \frac{L^{2}}{\epsilon^{2}}$, we have for $x \in \Theta$: 
\be
\|\delta_x \P^{t} - \mu_\Theta\|_{\TV} &\le \Pr[\thit^x(x) < C_{3}T \frac{L^{2}}{\epsilon^{2}}]  + \sup_{x \in [m-\epsilon, m+ \epsilon] \cap \Theta} \|\delta_x \P^{\tau_{[m-\epsilon,m+\epsilon]}(\P_{\Theta})} - \mu_\Theta\|_{\TV} \\
&\le \frac{1}{8} + \frac{1}{8} = \frac{1}{4}.
\ee
This completes the proof of the lemma.
\end{proof}

\subsection{Proof of Theorem \ref{ThmMainThm}}

Theorem \ref{ThmMainThm} follows quickly from our three main lemmas. Inequality \eqref{IneqMainThmConc1} is an immediate consequence of Lemmas \ref{Lemma2Stat} and \ref{Lemma3Stat}. Inequality \eqref{IneqMainThmConc2} is an immediate consequence of  Inequality \eqref{IneqMainThmConc1} and Lemma \ref{Lemma1Stat}.

\bibliographystyle{apalike}
\bibliography{lvmixing}

\end{document}